\documentclass [12pt]{article}
\usepackage[cp1251]{inputenc}
\usepackage{amsmath}
\usepackage{amssymb}
\usepackage{mathrsfs}
\usepackage{amsfonts, dsfont, srcltx}

\newtheorem{lemma}{Lemma}
\newtheorem{theorem}{Theorem}
\newtheorem{corollary}{Corollary}
\newtheorem{proposition}{Proposition}

\newtheorem{remark}{Remark}
\newcommand{\qed}{\vrule width 5pt height 5pt depth 0pt}

\newenvironment{proof}{\vspace{5pt}\bf Proof.
\rm}{\hfill\qed\vspace{5pt}}

\numberwithin{equation}{section}

\begin{document}

\thispagestyle{empty}

\begin{center}
{\Large Spectral synthesis for the differentiation operator in the Schwartz space}
\end{center}
\bigskip

\begin{center}{N.F. Abuzyarova\footnote{This work was supported by the Ministry of
 Education and Science of Russian Federation, project
no. 01201456408.}}
\end{center}


\begin{abstract} We consider the spectral synthesis problem for the differentiation operator $D=\frac{\text{d}\,}{\text{d}\, t}$ in the Schwartz space $\mathcal E(a;b)=C^{\infty} (a;b)$
and the dual problem of local description for closed submodules in a special module of entire functions.
\end{abstract}

\section{Introduction}

Let $(a;b)$ be a finite or infinite open interval of the real line and let $$[a_1;b_1]\Subset [a_2;b_2]\Subset \dots ,$$ 
be a sequence of segments exhausting this interval. We consider the Schwartz space $\mathcal E(a;b) =C^{\infty} (a;b)$, equipped by the metrizible topology of  projective limit of the Banach spaces  $C^k [a_k;b_k]$. It is known that  $\mathcal E(a;b)$ is a reflexive Fr\'echet space. Unless otherwise specified, we denote by $W$ a closed non-trivial subspace of $\mathcal E (a;b)$ which is invariant under the differentiation operator
 Shortly, we say such a  subspace $W$ to be $D$-{\it invariant}.

Let   $\mathrm{ Exp}\,   W$ denote the set of all exponential monomials $t^je^{-{\text i} \lambda t}$, that is, the set of all root elements  the  operator $D$ contained in $W$.
As we will see later, the set $\mathrm{ Exp}\,   W$ is at most countable.

A generalization of L. Euler
fundamental principle for finite order differential
equations with constant coefficients \cite{Euler}, as well as a generalization of classical results  on mean periodic functions due to L.Schwartz \cite{Schw-1},
is {\it the spectral synthesis} for a $D$-invariant subspace $W$:
\begin{equation}
W=\overline {{\mathcal L} (\mathrm{Exp}\, W)}, \quad\text{where} \quad {\mathcal L} (\ \cdot\ )  \quad\text{denotes the linear span of the set}
\quad (\ \cdot\ ).
\label{f2}
\end{equation}
 
A. Aleman and B. Korenblum notice in  \cite{Al-Kor} that there exist non-trivial $D$-invariant subspaces $W$ of $\mathcal E(a;b)$ with the property
$\text{Exp}\, W=\emptyset .$
Namely, let $I\subset (a;b)$ be a non-empty relatively closed interval and set
\begin{equation}
W_I=\{ f\in\mathcal E:\ \ f^{(k)}(t)=0,\ t\in I, \ k=0,1, 2, \dots\}.
\label{f1}
\end{equation}
 Obviously, the subspace $W_I$ is $D$-invariant and the set $\text{Exp}\, W_I$ is empty. In what follows, $W_I$ doesn't admit the  spectral synthesis (\ref{f2}).

For a $D$-invariant subspace $W ,$ we denote by $I_W$ a  relatively closed in  $(a;b)$ interval which is minimal   among all intervals $I$ with the property  $W_{I}\subset W.$ 
 The existence of $I_W$ follows from  \cite[Theorem 4.1]{Al-Kor}.
 
 In  \cite{Al-Kor}, its authors propose a weaker, than   (\ref{f2}), version of spectral synthesis for the differentiation operator $D$ in the space $\mathcal E(a;b).$ 
  We call it  {\it  spectral synthesis in the weak sense} (or, {\it  weak spectral synthesis}). This version takes into account the presence of  the {\it residual subspace} 
	$W_{I_W}\subset W.$
	$D$-invariant subspace $W$ is said to be 
 {\it admitting spectral synthesis in the weak sense} (or, {\it admitting weak spectral synthesis}) if
 \begin{equation}
W=\overline{W_{I_W}+{\mathcal L}(\mathrm{Exp}\,   W)}. 
\label{f3}
\end{equation}

The authors of the paper \cite{Al-Kor} prove that the {\it spectrum} $\sigma (W)$ of $D$-invariant subspace $W,$ which is defined to be the spectrum of the restriction
of the operator $D$ to  $W$, is either discrete or equal to the whole complex plane (\cite[Theorem 2.1]{Al-Kor}).  
In the first case, the spectrum $\sigma (W)$ is a  sequence of multiple points
\begin{equation}
\sigma (W) =\{ (-{\text i} \lambda_j, m_j)\ : \ \ \lambda_j\in\mathbb C,\ \  m_j\in \mathbb N, \ j=1,2,\dots\},
\label{specW}
\end{equation}
and    $\mathrm{Exp}\,   W =\{ t^ke^{-{\text i} \lambda_j t} \ : \ \ k=0,1,\dots , m_j-1, \ j\in\mathbb N \}.$ 
 
If the subspace $W$  has the form (\ref{f3}),
its spectrum is discrete and equals $(-\text{i}\,\Lambda)$, where
$\Lambda=\{ ( \lambda_j, m_j)\}.$ 
Such a subspace admits spectral synthesis (\ref{f2}) under the additional assumption $I_W=(a;b).$

 A. Aleman and B. Korenblum conjecture in  \cite[Section 6]{Al-Kor}
that any $D$-invariant subspace $W\subset \mathcal E(a;b) $ with discrete spectrum
has the form (\ref{f3}). 
They prove this conjecture for the case of
a subspace with finite (in particular, empty) spectrum \cite[Propositions 4.1, 6.1]{Al-Kor}

In \cite{ABB} this conjecture is studied by developing and improving methods of \cite{Al-Kor}  and  reducing the weak spectral synthesis problem to the  problem of completeness for  mixed systems in a Hilbert space.

To formulate the main result of the paper   \cite{ABB} 
we recall that {\it the completeness radius } $\rho_{\Lambda}$ of a sequence $\Lambda =\{ (\lambda_j, m_j)\}$ is defined as
 the infimum of the radii of (open) intervals $ I\subset\mathbb R $ for
which the system of exponential monomials $\{ t^ke^{-\text{i}\,\lambda_jt},\ \ k=0,\dots , m_j-1,\ \ j=1,2,\dots\} $ 
is not complete in the
spaces ${\mathcal E} (I),$ $C(I)$, $L^p(I),$ $1\le p<\infty$  (see \cite{BM}).

  Theorem 1.1 \cite{ABB} positively answers the question on the weak spectral synthesis for $D$-invariant subspace
 $W\subset \mathcal E(a;b)$,  with discrete spectrum of the form (\ref{specW}) if $2\rho_{\Lambda}<|I_W|$,
where $\Lambda =\text{i}\, \sigma (W)$ 
 and $|I_W|$ is the length of the interval $I_W$. It turns out that the relation (\ref{f3}) may  fail for 
$D$-invariant subspaces $W$ with discrete spectrum $\sigma (W)=-\text{i}\,\Lambda$ if  $2\rho_{\Lambda}=|I_W|.$ 
First example of such a subspace is constructed in\cite[Theorem 1.2]{ABB}. A lot of  examples, including the mentioned one, arise from  \cite[Theorem 3]{NF-UMJ-2}.
 In the other hand,  as our example in Section 3 below shows, $D$-invariant subspace $W$ with discrete spectrum $\sigma (W)=-\text{i}\,\Lambda$ satisfyng
the relation $2\rho_{\Lambda}=|I_W|$, may also be of the form (\ref{f3}).

We announce in \cite{NF-DAN} the  necessary and sufficient condition for $D$-invariant subspace $W\subset \mathcal E(a;b)$ to be admitting weak spectral
 synthesis.
 Our approach to the problem does  	completely differ from the methods used in the works \cite{Al-Kor}, \cite{ABB}. Namely, we apply
		a dual scheme going back to I.F Krasichkov-Tenovskii \cite{IF-synth}  and L. Ehrenpreis \cite{Ehren-synth}. 
This scheme reduces  the problems concerning $D$-invariant subspaces of the space $\mathcal E(a;b)$ to the equivalent ones on closed submodules in a special weighted module of entire functions. 
We formulate the mentioned  criterion  in terms of the annihilator submodule of  $D$-invariant subspace $W$
\cite[Theorem 3]{NF-DAN}.
Then, we deduce from it different sufficient conditions of admitting of weak spectral synthesis by $D$-invariant subspace $W$ with discrete spectrum 
(\cite[Corollaries 1 and 2]{NF-DAN}, Theorem \ref{teorema9} below,  \cite[Theorems 1.1, 1.3]{ABB}).

This paper contains the explicit proofs for  most of the assertions announced in \cite{NF-DAN}  concerning weak spectral synthesis in the space $\mathcal E(a;b)$.
In addition, we prove one theorem on weak spectral synthesis for ''translation-invariant'' subspaces of  $\mathcal E (a;b)$ (Theorem \ref{teorema9} below).

The organization of the paper  is as follows. 
In Section 2 we describe the mentioned dual scheme.
Section 3 contains  a series of  results on  submodules of entire functions.
At last,
in Section 4  we derive  assertions on  $D$-invariant subspaces in $\mathcal E(a;b)$.

\section{Dual scheme}

{\bf 1.}
Let us define a locally convex space of entire functions  $\mathcal P (a;b)$ 
to be equal to  inductive limit of the sequence of the Banach spaces
 $\{ P_k\}$, where    $P_k$ consists of all entire functions $\varphi$ 
 with finite norm
$$
\| \varphi\|_k =\sup_{z\in\mathbb C} \frac{|\varphi (z)|}{(1+|z|)^k\exp (b_ky^{+}-a_ky^{-})}, \quad
y^{\pm}=\max\{ 0,\pm y\}, \quad z=x+{\text i} y .
$$
The inclusions $P_k\subset P_{k+1}$ are compact for each $k=1,2,\dots$ It follows that the space   
$\mathcal P(a;b)$ is a locally convex space of $(LN^*)$-type (see \cite{Seb}).
 Further, the operation 
 of multiplication by an independent variable $z$ acts
continuously in this space, that is, $\mathcal P(a;b)$ is a topological module over the
polynomial ring  $\mathbb C[z].$

It is well-known that 
the Fourier–Laplace transform $\mathcal F$ 
establishes a linear topological isomorphism between
the strong dual space
 $\mathcal E'(a;b)$ and the space $\mathcal P (a;b)$ \cite[chapter 7, theorem 7.3.1]{Horm}:
$$
S\in\mathcal E'(a;b) \longleftrightarrow \varphi\in{\mathcal P(a;b)}
\Longleftrightarrow \varphi =\mathcal F (S)=(S,e^{-{\text i} \, tz}).
$$

Let  $\mathcal J \subset\mathcal P(a;b)$ be a closed {\it submodule,} that is,  $\mathcal J$
is a closed subspace with the property $z\mathcal J\subset \mathcal J.$ 
Two natural characteristics of the submodule   $\mathcal J$
arise immediately: its {\it zero set} $\Lambda_{\mathcal J}$
and its {\it indicator segment}  $[c_{\mathcal J}; d_{\mathcal J}]$. The first one, $\Lambda_{\mathcal J}$,
is well-known. We always meet it when dealing with the  local description problem for ideals and submodules of holomorphic functions.
The notion of the  indicator segment  is introduced in  \cite{NF-DAN}, \cite{NF-UMJ-1}.
Below we recall both definitions.
Given a function  $\varphi\in \mathcal P(a;b),$  we set at any $\lambda\in\mathbb C$ 
  $$
n_{\varphi} (\lambda)=\begin{cases} 0,\quad \text{if} \quad\varphi(\lambda)\neq 0,\\
m,\quad\text{if}\quad\lambda \ \ \text{is a zero of}\quad\varphi
\quad\text{of muliplicity}\quad m.
\end{cases}
$$
 {\it  The divisor of a closed
submodule} $\mathcal J\subset \mathcal P(a;b)$ is defined as $n_{\mathcal J}(\lambda)=\min\limits_{\varphi\in\mathcal J} n_{\varphi}(\lambda)$.

The {\it zero set of a non-zero function} $\varphi\in\mathcal P(a;b)$ is
 $$\Lambda_{\varphi} =\{ (\lambda_k, m_k) : \ m_k=n_{\varphi }(\lambda_k)>0 , \  k=1,2,\dots\}$$
The {\it  zero set
of a submodule} $\mathcal J\neq\{ 0\}$
is 
  $$\Lambda_{\mathcal J} =\{ (\lambda_k, m_k) : \ m_k=n_{\mathcal J }(\lambda_k)>0 , \  k=1,2,\dots\}$$

As it is known (see, e.g., \cite{Levin}), any element of the space $\mathcal P(a;b)$ is a
function of completely regular growth with respect to
the order 1, and the  indicator diagram of any $\varphi\in\mathcal P(a;b)$ is a segment $[{\text i} c_{\varphi};{\text i} d_{\varphi}] \subset ({\text i} a;{\text i} b).$  
The  {\it indicator segment} $[c_{\mathcal J}; d_{\mathcal J}] $ of a submodule $\mathcal J$ is defined to be the segment (generally speaking, in $\overline{\mathbb R}$)
  with the endpoints 
\begin{equation}
c_{\mathcal J}=\inf\limits_{\varphi\in\mathcal J}c_{\varphi},\quad
d_{\mathcal J}=\sup\limits_{\varphi\in\mathcal J}d_{\varphi}.
\label{ind-segm}
\end{equation}

For a closed subspace   $W\subset \mathcal E(a;b)$  we set
 $$W^0=\{ S\in{\mathcal E'(a;b)} :\ \
(S,f)=0, \ \ f\in W\}.$$
This is a closed subspace of the strong dual space $\mathcal E'(a;b)$. It  is said to be {\it the annihilator subspace} of   $W$.

The following proposition establishes the duality between $D$-invariant subspaces $W\subset \mathcal E(a;b)$ 
and closed submodules $\mathcal J\subset\mathcal P(a;b)$
\cite[proposition 1]{NF-UMJ-1}.
\begin{proposition}
({\rm Duality principle.}) Between the set  $\{W\}$ of all $D$-invariant subspaces of the space $\mathcal E(a;b)$ and the set  $\{\mathcal J\}$ of all closed submodules of the module 
$\mathcal P (a;b)$, there
is a one-to-one correspondence, namely,
$$W\longleftrightarrow \mathcal J \Longleftrightarrow  {\mathcal J}={\mathcal F}(W^0).$$ 
Moreover,
\begin{equation}
\mathrm{Exp}\,   W=\{ t^je^{-{\text i}\lambda_k t}\ : \ \ j=0,\dots m_k-1,\ \ (\lambda_k,m_k)\in\Lambda_{\mathcal J}\},
\label{eW}
\end{equation}
 and the endpoints of the interval
$I_W$ are $c_{\mathcal J}$ and $d_{\mathcal J}$.
	\label{DP}
	\end{proposition}

Obviously,  $D$-invariant  subspace $W$  admits weak spectral synthesis if and only if it is the minimal one among all $D$-invariant subspaces $\tilde{W}$ with
properties
	$$
	I_{\tilde{W}}=I_W,\quad \text{Exp}\, \tilde{W}=\text{Exp}\, {W}.
	$$
	By the duality principle (proposition \ref{DP}), we 
	see that the 
	 {\it annihilator submodule } $\mathcal J=\mathcal F (W^0)$ of such a subspace $W$
	 should be maximal one among all 
submodules $\tilde{\mathcal J}\subset \mathcal P(a;b)$ such that
	$$
	\Lambda_{\tilde{\mathcal J}}=\Lambda_{\mathcal J},\quad [c_{\tilde{\mathcal J}};d_{\tilde{\mathcal J}}] =[c_{\mathcal J};d_{\mathcal J}].
	$$
	If so, we say the submodule $\mathcal J$ to be {\it weakly loclizable}. In the other words,  submodule $\mathcal J\subset \mathcal P(a;b)$ is  weakly localizable
	if it contains any function $\psi\in\mathcal P(a;b),$ which vanishes (with the  multiplicities)
	on the   set $\Lambda_{\mathcal J}$ and has the  indicator diagram 
	$[\text{i} c_{\psi}; \text{i} d_{\psi}] \subset [\text{i} c_{\mathcal; J};\text{i} d_{\mathcal J}].$

As a conclusion we get 	
		\begin{proposition}
	$D$-invariant subspace $W\subset \mathcal E(a;b)$ admits weak spectral synthesis if and only if its annihilator submodule 
	$\mathcal J=\mathcal F (W^0)\subset\mathcal P(a;b)$
	is weakly localizable.
	\label{prop1}
\end{proposition}

Let us notice that
the weak localizable submodule $\mathcal J=\mathcal F (W^0)$ turns to be  {\it localizable} or {\it ample} (see \cite{IF-loc-1}, \cite{IF-loc-2}) 
if 
\begin{equation}
I_W=[c_{\mathcal J};d_{\mathcal J}]\bigcap (a;b) = (a;b).
\label{class-s}
\end{equation}
 Under the assumption \ref{class-s}, Proposition \ref{prop1} becomes the classical duality principle
(cf. \cite{IF-synth}).

Unless otherwise specified,  we consider only closed submodules
$\mathcal J \subset\mathcal P(a;b).$

A submodule $\mathcal J$ is said to be  {\it stable at a point} $\lambda\in\mathbb C$
if for any $\varphi\in\mathcal J$  the inequality $n_{\varphi} (\lambda)>n_{\mathcal J}(\lambda)$ implies  $\varphi/ (z-\lambda )\in\mathcal J.$ A submodule $\mathcal J$  is said to be {\it stable} if it is stable at any point $\lambda\in\mathbb C.$
These notions are   introduced 
in \cite{IF-synth}, \cite{IF-loc-1}.

The propositions \cite[Proposition 3.1]{ABB}  and \cite[Proposition 2]{NF-UMJ-1}, taking together, can be expressed as the following assertion.
\begin{proposition}
The spectrum of $D$-invariant subspace $W\subset\mathcal E(a;b)$ is discrete if and only if its annihilator submodule $\mathcal J$ is stable.
\label{prop2}
\end{proposition}

\section{Weakly localizable submodules in $\mathcal P(a;b)$.}

{\bf 1.}
 Proposition \ref{prop1} reduces the problem of weak spectral synthesis for $D$-invariant subspace $W$
to the studying  of weak localizability of its annihilator submodule $\mathcal J$.
One can easy see that a weakly localizable submodule is necessarily stable.

Below we recall the  definitions of some notions we will need in the further considerations. These notions have been introduced in \cite[ \S 1 and \S 3, remark 2  ]{IF-loc-1}.
 We give them  for the case of scalar functions.

A submodule $\mathcal J$ is {\it $b$-saturated with respect to a function} $\psi\in\mathcal P(a;b)$ if there exists a bounded set $B\subset\mathcal P(a;b)$ sucn that the following implication is true:
if  $\Phi (z)$ is an entire function and the inequality
$
|\Phi (z)\varphi (z)|\le |\varphi (z)|+|\psi (z)|  
$
 holds each $z\in\mathbb C$ and for any function  $\varphi\in B\bigcap\mathcal J,$ 
then $\Phi =const .$ 

A locally convex space $P$ of entire functions is said to be
 {\it $b$-stable} if, for any bounded set $B\subset\mathcal P$, 
the set of all entire  
functions $\psi$ of the form $$ \psi (z)=\varphi (z)/(z-\lambda),\quad \lambda\in\mathbb C,  \quad\varphi\in B,$$
is bounded in $P$. 

As we have noticed above,  $\mathcal P(a;b)$ is a $(LN^*)$-type space. It follows that a subset  $B\subset \mathcal P(a;b)$ 
is bounded if and only if, for some $k,$ it is contained and bounded in the  Banach space  $P_k$  \cite[Theorem 2]{Seb}.
Using this fact and the definition of the topology in  $\mathcal P(a;b)$ we can easily  verify that the space
 $\mathcal P(a;b)$ is a bornologic and  $b$-stable. 
For spaces  of holomorphic vector-functions, which are bornologic and  $b$-stable, I.F. Krasichkov-Ternovskii proved the following theorem \cite{IF-loc-1}.
(We cite it here for the case of our space  $\mathcal P(a;b)$.)

{\sl Bornological Version of Individual Theorem.}  {\it Let $\mathcal J$ be a stable submodule in  $\mathcal P(a;b)$,
 and let $\psi\in \mathcal P(a;b)$ be a function satisfying the following assumption: $n_{\psi}(z)\ge n_{\mathcal J}(z)$ in the whole complex plane. Then,  $\psi\in\mathcal J$
if and only if  $\mathcal J$ is $b$-saturated with respect to $\psi.$}

This  assertion   allows us to obtain the criterion of weak localizibilty for closed submodules of the module $\mathcal P(a;b)$.

Given a  function $\varphi\in\mathcal P(a;b),$  we denote by ${\mathcal  J}(\varphi)$ the submodule formed by all functions $\psi=\omega\varphi \in\mathcal P(a;b) .$ 
 Here  $\omega$ is an entire function of mi\-ni\-mal exponential type with the integral $\int\limits_{-\infty}^{\infty} \left(\log^+|\omega(x)|/(1+x^2)\right){\mathrm d} x$ to be converging. In the  other words, ${\mathcal  J}(\varphi)$ consists of all functions $\psi\in\mathcal P(a;b) $, which are dividing by  $\varphi$ and  have the same indicator diagram as $\varphi.$ 
Clearly, the submodule ${\mathcal  J}(\varphi)$ is weakly localizable.  
  
To begin with, we prove one auxiliary assertion.

\begin{lemma}
Let $\mathcal J$ be a stable submodule.
Suppose that for a function $\varphi\in\mathcal P(a;b)$ the inequality $n_{\varphi}(z)\ge n_{\mathcal J}(z)$ holds
 for any $z\in\mathbb C,$ and its indicator diagram $[\text{i}c_{\varphi};\text{i}d_{\varphi}]$ is a subset of the open interval $(\text{i}c_{\mathcal J};\text{i}d_{\mathcal J})$. 

Then $J(\varphi)\subset \mathcal J.$
\label{lem-help}
\end{lemma}  

\begin{proof}
Let $\tilde{\varphi}$ be an arbitrary function in $\mathcal J(\varphi).$
By assumption, we have $c_{\mathcal J}<c_{\varphi}$. Using the definition of $c_{\mathcal J}$ we can find
a function $\varphi_1\in\mathcal J$
with the property $$c_{\mathcal J}\le c_{\varphi_1}<c_{\varphi} .$$
We can also find another function $\varphi_2\in\mathcal J$ satisfying the inequalities
$$
d_{\varphi}<d_{\varphi_1}\le d_{\mathcal J}.
$$
Set  $\varphi_B =\varphi_1+\varphi_2.$
 It follows from the definition of the submodule $\mathcal J(\varphi)$ that
the indicator diagram of an arbitrary function $\tilde{\varphi}\in\mathcal  J(\varphi)$ is the segment
$[\mathrm{i} c_{\varphi};\mathrm{i} d_{\varphi}]$. Therefore, this indicator diagram is a compact subset of the indicator diagram $[\mathrm{i} c_{\varphi_B};\mathrm{i} d_{\varphi_B}] .$ 
Taking into the account that the function $\varphi_B$ has  completely regular growth,  we get 
\begin{equation}
\frac{\tilde {\varphi} (z)}{\varphi_B (z)}\to 0,
\label{tend}
\end{equation}
where $z=re^{i\theta}$ and $r$ tends to $\infty$ lying outside a set of zero relative measure.
Moreover, the relation (\ref{tend}) holds
uniformly on $\theta \in \{ |\pi/2 -\theta|<\delta\}\bigcup \{|-\pi/2-\theta|<\delta\}$
if   $\delta>0$ is sufficiently small.

Now, it is not difficult to show that the submodule $\mathcal J$ is $b$-saturated with respect  to any function
 $\tilde{\varphi}\in\mathcal J(\varphi).$
For this purpose we set $B=\{\varphi_B\}$ and consider an arbitrary entire function $\omega$ satisfying  
the inequality
\begin{equation}
|\omega(z)\varphi_B(z)|\le |\tilde{\varphi} (z) | +|\varphi_B (z)| ,\quad z\in \mathbb C.
\label{satur}
\end{equation}
From the relations (\ref{tend}), (\ref{satur}), by the maximum modulus principle, we derive that 
the function  $\omega$ is bounded on the imaginary axis.
The relation (\ref{satur}) and the completely regular growth of the function $\varphi_B$
imply that the function $\omega$ has minimal exponential type.
In what follows, $\omega=const$ 
 and the stable submodule $\mathcal J$ is $b$-saturated with respect to the function $\tilde{\varphi}.$
Applying the above cited Bornological Version of Individual Theorem, we conclude that 
$\tilde{\varphi}\in \mathcal J$.

\end{proof}

\begin{theorem}  Assume that $\mathcal J$ is a stable submodule. This submodule is weakly localizable if and only if
there exists a function $\varphi\in \mathcal J $ with the property 
			\begin{equation}
			{\mathcal J}(\varphi)\subset {\mathcal J}.
			\label{pr-fi}
			\end{equation}
      \label{teorema1}
      \end{theorem}

\begin{proof}
We should notice that the assertion is trivial when $c_{\mathcal J}=d_{\mathcal J}$.
Indeed, in this case $\mathcal J$ is exactly the set  $\{\alpha e^{-\text{i}\, c_{\mathcal J}z}\},$
$\alpha\in\mathbb C.$
In what follows, we  assume that 
$c_{\mathcal J}<d_{\mathcal J}.$

Clearly, we only need to prove the sufficient part of the theorem:  the existence of a function $\varphi\in\mathcal J$ with the property (\ref{pr-fi}) implies that the submodule $\mathcal J$ is weakly localizable.

1) First, we suppose that (\ref{pr-fi}) holds for a function $\varphi$ which  indicator diagram coincides with to the segment $[\text{i} c_{\mathcal J};\text{i} d_{\mathcal J}].$

Given an arbitrary function  $\psi\in \mathcal P(a;b)$ with  zero set $\Lambda_{\psi}\supset \Lambda_{\mathcal J}$
and  indicator diagram contained in $[\text{i}c_{\mathcal J};\text{i}d_{\mathcal J}],$
we will verify that the submodule $\mathcal J$ is $b$-saturated with respect to the function $\psi.$
Set 
$$
B =\{ \tilde{\varphi}\in Hol(\mathbb C): \ \ |\tilde{\varphi}(z)|\le |\psi (z)|+|\varphi (z)|,\ \ z\in\mathbb C \}.
$$
 From the topological  properties of the space $\mathcal P(a;b)$,
 it follows that  $B$ is a bounded subset of $\mathcal P(a;b)$.
Let $\omega $ be an entire function  satisfying the inequality
\begin{equation}
|\omega (z)\tilde{\varphi}(z)|\le |\tilde{\varphi} (z)|+|\psi (z)|
\label{satur1}
\end{equation}
in the whole complex plane
for any function $\tilde{\varphi}\in B\bigcap \mathcal J$.
 In particular, this inequality is true  when $\tilde{\varphi}=\varphi $.
In what follows,  $\omega$ is a function of minimal exponential type and  $\omega\varphi\in 
\mathcal J(\varphi).$
The inclusion
$\mathcal J(\varphi)\subset \mathcal J ,$ (\ref{satur1}) and the definition of the set $B$,
lead to the inclusion 
 $\omega\varphi \in B\bigcap \mathcal J.$
Setting $\tilde{\varphi} =\omega\varphi$ in  (\ref{satur1}), we get  the inequalities 
$$
\left|\omega^2 (z) \varphi (z)\right|\le \left|\omega (z){\varphi} (z)\right|+|\psi (z)|\le 2 (|\varphi (z)|+|\psi (z)|) ,
$$
for any  $z\in\mathbb C$.
Further, arguing as we have  done  for the function $\omega\varphi$, we derive that
$$
\frac{\omega^2}{2}\varphi\in B\bigcap \mathcal J.
$$
 Now, setting $\tilde{\varphi} =\frac{\omega^2}{2}\varphi$ in  (\ref{satur1}) gives us
 the relations
$$
|\frac{\omega^3 (z)}{2} \varphi (z)|\le |\frac{\omega^2 (z)}{2}{\varphi} (z)|+|\psi (z)|\le 2 (|\varphi (z)|+|\psi (z)|) 
$$
in the whole complex plane.
 In what  follows, the inclusion $$\frac{\omega^3}{2^2}\varphi \in B\bigcap \mathcal J$$
is valid.

Continuing to argue by  the similar way,   we  obtain  that
\begin{equation}
\frac{|\omega^n (z)|}{2^{n-1}} |\varphi(z)|\le |\varphi (z)|+|\psi (z)|
\label{ineq-n}
\end{equation}
 for each
 $n\in\mathbb N$ and for any $z\in\mathbb C.$
The only case when it can happen is  $\omega =const .$
Hence, the submodule $\mathcal J$ is $b$-saturated with respect to the function $\psi .$
Applying  Bornological Version of Individual Theorem we get the  inclusion 
$\psi\in \mathcal J$. This is true for any function $\psi\in\mathcal P(a;b)$ satisfying $\Lambda_{\psi}\supset\Lambda_{\mathcal J}.$
Therefore, we conclude that the submodule $\mathcal J$ is weakly localizable.

\medskip

2) Let us now consider the case when $\mathcal J(\varphi)\subset \mathcal J$ and the indicator diagram $[\text{i} c_{\varphi};\text{i} d_{\varphi}]$ of the function $\varphi$ is a proper subset of the segment $[\text{i} c_{\mathcal J};\text{i} d_{\mathcal J}].$
We also assume in this part of the proof that $[\text{i} c_{\mathcal J};\text{i} d_{\mathcal J}] \subset  (\text{i} a;\text{i} b)$.
The restriction on the indicator diagram   $[\text{i} c_{\varphi};\text{i} d_{\varphi}]$ implies that at least one of the values  $\delta_1 =c_{\varphi}-c_{\mathcal J}\ge 0,$ $\delta_2 =d_{\mathcal J}-d_{\varphi}\ge 0$
is strictly positive.
Suppose, e.g., that  $\delta_1>0$ and $\delta_2>0.$ Then, by Lemma \ref{lem-help}, we get   $\mathcal J(e^{i\delta' z}\varphi )\subset\mathcal J$
for any $\delta' \in [0;\delta_1)$ and $\mathcal J(e^{-i\delta'' z}\varphi),$ for any $\delta ''\in
 [0;\delta _2).$  
In particular, 
 \begin{equation}
 e^{i\delta ' z}\varphi , \ \ e^{-i\delta '' z}\varphi \in\mathcal J \quad \text{for any}\quad \delta'\in [0;\delta_1),\  \delta''\in [0;\delta_2). 
 \label{incl}
 \end{equation}
  
 Note that the relations
$$
\lim_{\delta'\to\delta_1}e^{i\delta ' z}\varphi =e^{i\delta_1 z}\varphi ,
\quad
\lim_{\delta''\to\delta_2}e^{i\delta '' z}\varphi =e^{i\delta_2 z}\varphi ,
$$ 
are valid in the topology of $\mathcal P(a;b). $
These relations, together with (\ref{incl}), lead to the inclusion  $\Phi \in\mathcal J,$
where $\Phi = (e^{i\delta_1 z}+e^{-i\delta_2 z})\varphi$.
  Further, any function $\Psi\in\mathcal J(\Phi)$ has the form 
 $$\Psi=\omega \Phi =\omega (e^{i\delta_1 z}+e^{-i\delta_2 z})\varphi ,$$ 
 where $\omega$ is an entire function of minimal exponential  type.
There is  no difficulty to verify that  $\omega\varphi\in\mathcal P (a;b)$, and  Lemma \ref{lem-help} implies the inclusions
 $$\omega\varphi\in\mathcal J,
\quad
e^{i\delta ' z} \omega\varphi \in\mathcal J, \ \ e^{-i\delta '' z}\omega\varphi \in\mathcal J 
 $$
for each $\delta'\in (0;\delta_1)$ and for each $  \delta''\in (0;\delta_2)$.
We can easily see that   
$$ e^{i\delta ' z} \omega\varphi+ e^{-i\delta '' z}\omega\varphi \to \Psi \quad\text{as}\quad \delta'\to\delta_1,\ \ \delta''\to\delta_2$$
in the topology of $\mathcal P(a;b),$
and, consequently,  $\Psi\in\mathcal J$.
 Finally, we obtain the inclusion $\mathcal J(\Phi)\subset \mathcal J.$

The indicator diagram of the function $\Phi$ is equal to the segment
$[\text{i} c_{\mathcal J}; \text{i} d_{\mathcal J}]$. By this observation and  the first part of the proof, we conclude that the submodule $\mathcal J$ is weakly localizable.

If  one of the values  $\delta_1$  or $\delta_2$ equals  zero we should argue by  the  similar way, with obvious  changes.
\medskip

3) It remains to consider the case when at least one  of the endpoints of the indicator segment $[ c_{\mathcal J}; d_\mathcal {J}]$ coincides with the corresponding endpoint of the interval $(a;b)$
(that is, $c_{\mathcal J}=a$ or  $d_{\mathcal J}=b$).

 Given any function $\Psi\in\mathcal P(a;b) $  with the indicator diagram $[\text{i} c_{\Psi};\text{i} d_{\Psi}]\subset [\text{i}c_{\mathcal J};\text{i}d_{\mathcal J}]$ 
and the zero set $\Lambda_{\Psi}\supset\Lambda_{\mathcal J}$, we should  prove  the inclusion $\Psi\in\mathcal J.$
Let  $[c';d']$ be a segment subjected the following restrictions:
\begin{equation}
[c';d']\subset (a;b)\bigcap [c_{\mathcal J};d_{\mathcal J}],\quad
[c_{\Psi};d_{\Psi}]\subset [c';d'],\quad [c_{\varphi};d_{\varphi}]\subset [c';d'].
\label{i-segm}
\end{equation}
We denote by $\mathcal J'$ a weakly localizable submodule 
with the zero set  $\Lambda_{\mathcal J'}=\Lambda_{\mathcal J}$ and the indicator segment 
$[c';d']$.
Then, 
$\tilde{\mathcal J} =\mathcal J\bigcap\mathcal J'$ is a stable submodule
with the zero set $\Lambda_{\tilde{\mathcal J}}
=\Lambda_{\mathcal J}$ and the indicator segment $[c';d'].$ 

The relations  (\ref{i-segm}) lead to the inclusion $\mathcal J (\varphi)\subset\tilde{\mathcal J}.$
Therefore, by the  previous parts of the proof, we get that  $\tilde{\mathcal J} =\mathcal J'.$ 
Taking into  account the relations (\ref{i-segm}) one more time, we conclude that
$\Psi\in\tilde{\mathcal J}\subset \mathcal J.$
\end{proof}

\begin{remark}
  In the first part of the    proof  we use the argument which is similar to the one has been used by
 I.F. Krasichkov-Ternovskii 
 \cite[\S 11, Theorem 11.1]{IF-synth-3}, \cite[\S 5, Proposition 5.5]{IF-loc-2}. 
\end{remark}

Given a function $\varphi\in\mathcal P(a;b),$ we, as usual, 
define the {\it principal} submodule generated by  $\varphi$
to be  the closure of the set $\{ p\varphi :\ p\in\mathbb C[z]\}$
 in  $\mathcal P(a;b). $
 We denote this submodule  by $\mathcal J_{\varphi}.$

\begin{theorem}
Let  $\mathcal J$ be a stable submodule. Each of the assumptions listed below is sufficient for the submodule $\mathcal J$ to be weakly localizable: 

\noindent
1) there exists a function
$\varphi\in\mathcal J$
 which generates the weakly localizable principal submodule $\mathcal J_{\varphi}$;

\noindent
2)  the inequality
\begin{equation}
2\rho_{\Lambda_{\mathcal J}}< d_{\mathcal J}-c_{\mathcal J}
\label{rad-compl}
\end{equation}
holds, where $d_{\mathcal J}-c_{\mathcal J}=+\infty $ if at least one of the values  $c_{\mathcal J}$ or  $d_{\mathcal J}$ is not finite.
\label{teorema2}
\end{theorem}

\begin{proof}
1) Let $\varphi\in\mathcal J$ be a function generating the weakly localizable principal submodule $\mathcal J_{\varphi}.$ Then, $\mathcal J_{\varphi}=\mathcal J(\varphi)\subset\mathcal J,$ and  Theorem \ref{teorema1} implies  that the submodule $\mathcal J$ 
is weakly localizable.

2) Suppose that the inequality (\ref{rad-compl}) holds.
According to the definition of $\rho_{\Lambda}$ and  \cite[chapter 7, Theorem 7.3.1]{Horm},
there exists a function  
$\varphi \in\mathcal P(a;b)$ with the zero set  $\Lambda_{\varphi}\supset\Lambda_{\mathcal J}$ and the indicator diagram
$[\text{i} c_{\varphi};\text{i} d_{\varphi}]\subset (\text{i} c_{\mathcal J};\text{i}  d_{\mathcal J}).$
By Lemma \ref{lem-help} we get  the inclusion $\mathcal J(\varphi)\subset \mathcal J.$
Applying Theorem \ref{teorema1} we conclude that the submodule $\mathcal J$ is weakly localizable.

\end{proof}

\begin{corollary}
A stable submodule  $\mathcal J$ is ample (localizable) if and only if its indicator segment is equal to the segment $[a;b]\subset [-\infty;+\infty] .$
\label{sled1}
\end{corollary}

\begin{remark}
 {\rm To study the weak localizability of  stable submodules   by verifying  the assumption 1) of Theorem \ref{teorema2} is only reasonable   when
$2\rho_{\Lambda_{\mathcal J}}= d_{\mathcal J}-c_{\mathcal J}.$  
If so, we have $[c_{\mathcal J};d_{\mathcal J}]\subset (a;b)$, and the indicator diagram of any function $\psi\in \mathcal J$ is equal to $[\mathrm{i} c_{\mathcal J}; \mathrm{i} 
 d_{\mathcal J}]$.

A function $\varphi\in\mathcal P(-\infty ;+\infty)$
is said to be {\it invertible} (see \cite{Ber-Tayl}) if for any $\Phi\in \mathcal P(-\infty ;+\infty)$ the function  
$\Phi/\varphi$ belongs to   $ \mathcal P(-\infty ;+\infty )$ whenever it is  an entire one.

The assumption 1) of Theorem \ref{teorema2} is valid if the submodule $\mathcal J$ contains an invertible function or, more generally, a function $\varphi$ satisfying  the  relations 
\begin{equation}
\mathcal J_{\varphi}=\mathcal J(\varphi)=\{ p\varphi:\ \ p\in\mathbb C[z]\}.
\label{soot}
\end{equation}
Clearly, the relations (\ref{soot}) hold for an invertible function $\varphi$.
 \cite[Theorem 1]{NF-UMJ-2}  shows that they may also be valid when the function $\varphi $ is not  invertible.
 
Let  $\varphi\in\mathcal P(a;b)$ be such a function that the  submodule
 $\mathcal J(\varphi)$ contains at least one function $\Phi =\omega\varphi,$ where $\omega$ is not  a polynomial.
 Then, as it is proved in \cite[Theorem 2]{NF-UMJ-2},  the principal submodule $\mathcal J_{\varphi}$ may be weakly localizable only if its generator $\varphi$ is an element of the space $\mathcal F (C_0^{\infty} (a;b))$.
  In the other hand, it follows from   \cite[Theorem 1.2]{ABB} and \cite[Theorem 3]{NF-UMJ-2}, 
	the inclusion 
	 \begin{equation}
 \varphi\in\mathcal F (C_0^{\infty} (a;b))
 \label{necess}
 \end{equation} 
	is not a sufficient condition 
	for the principal submodule $\mathcal J_{\varphi}$
to be  weakly localizable. 
 
 To get any criteria of weak localizability for the principal submodule $\mathcal J_{\varphi}$ in terms of  characteristics
 of its generator  $\varphi$ seems to be quite difficult task. It can be reduced to the equivalent problem
 of weighted polynomial approximation. We are going to study it later in other place. Here we confine ourselves to considering  an example  of  weakly localizable principal submodule generated by a function $ \varphi\in\mathcal F (C_0^{\infty} (a;b))$.}
  \end{remark}

{\bf Example.}  
Suppose that $a<-\pi$ and $b>\pi$.
Then, the functions
\begin{eqnarray}
s(z)=\frac{\sin \pi z}{\pi z},\quad s_1(z)=s(\sqrt{z} ),\\
 \varphi (z)=\frac{s(z)}{s_1(z)s_1(-z)}
\label{def-fi}
\end{eqnarray}
belong to the module $\mathcal P(a;b)$.

We assert that {\it the principal submodule $\mathcal J_{\varphi}\subset\mathcal P(a;b)$ is weakly localizable.}

It is well-known that 
\begin{eqnarray}
|s(z)|\le \frac{c_0e^{\pi |\mathrm{Im}\, z|}}{\pi (1+ |z|)}, \quad z\in\mathbb C, 
\label{s-1}\\
|s (z)|\ge \frac{m_{d}e^{\pi|\mathrm{Im}\, z|}}{\pi |z|},\quad |z-k|\ge d, \quad k\in\mathbb Z,
\label{s-2}
\end{eqnarray}
where $c_0$ is an absolute constant, $d\in (0;1/2)$, and  $m_d$ is a positive constant depending on $d.$ 
The estimates (\ref{s-1}) and (\ref{s-2}) imply that the relations (\ref{soot}) are valid for the submodules $\mathcal J_s$ and $\mathcal J(s).$
By Theorem \ref{teorema2},  the desired assertion will follow from the inclusion
\begin{equation}
s\in\mathcal J_{\varphi}.
\label{sufi}
\end{equation}

To prove (\ref{sufi}) we should approximate the function $s$ in the space $\mathcal P(a;b )$ by functions of the form $p\varphi ,$ where $p$ is a polynomial.
Set $\omega (z) =s_1(z)s_1(-z)$.
 It follows from (\ref{s-1}) that for any real $x$ the inequalities
\begin{equation}
|\omega (x)|\le \frac{c_0^2 e^{\pi\sqrt{|x|}}}{\pi^2 (1+\sqrt{|x|})^2}\le c_1 e^{\pi\sqrt{|x|}} 
\label{om-above}
\end{equation}
hold, where
 $c_1=\frac{c_0^2}{\pi^2} .$

We are going to apply the following lemma (see \cite[Lemma 1]{NF-UMJ-2}) to get a proper  estimate for the function $\varphi$ on the real axis.
\begin{lemma}
Let  $d_0\in (0;1/2)$ be so small that $\left|\frac{\sin\pi\xi}{\pi\xi} -1\right|\le 1/2$ for any $\xi\in\mathbb C$ satisfying  $\pi|\xi|\le d_0.$
Then, there exists a positive  constant $c_{d_0}$ depending on $d_0$
such that the inequality
\begin{equation}
|s_1(z)|\ge
\frac{c_{d_0}e^{\pi\sqrt{|z|}|\sin(\theta /2)|}}{1+|z|},
\label{s-6}
\end{equation}
holds if $ z\in\mathbb C\setminus \bigcup\limits_{k\in\mathbb N}\left\{z:\, |z-k^2|< 3d_0\right\} .$
\label{lem0} 
\end{lemma}

From the relations (\ref{def-fi}), (\ref{s-1}), (\ref{s-6}), we derive the estimates for the function $\varphi:$ 
\begin{multline*}
|\varphi (z)|\le \frac{c_0 (1+|z|)}{\pi c_{d_0}^2 e^{\pi\sqrt{|z|}}}  \\ \text{if} \quad z\in\mathbb R
\setminus
\bigcup\limits_{k\in\mathbb N}\left(\left\{z:\, |z-k^2|< 3d_0\right\}\bigcup\left\{z:\, |z+k^2|< 3d_0\right\}\right),
\\
|\varphi (z)|\le \frac{c_0 e^{3\pi d_0}(1+|z|)}{c_{d_0}^2\pi e^{\pi\sqrt{|z| }(|\sin (\theta /2)|+|\cos (\theta /2)|)}}
\\ \text{if}\quad  z\in \bigcup\limits_{k\in\mathbb N}\left(\left\{z:\, |z-k^2|= 3d_0\right\}\bigcup\left\{z:\, |z+k^2|= 3d_0\right\}\right).
\end{multline*} 
Applying the maximum modulus principle,  we get 
\begin{equation}
|\varphi (x)|\le C_{1} (1+|x|)e^{-\pi\sqrt {|x|}}
\label{fi-ab-2}
\end{equation}
for any $x\in\mathbb R$, where $C_1$ depends only on $d_0\in (0;1/12).$

We set $W(x)=(1+|x|)e^{\pi\sqrt{|x|}}$ to be a weight function.
By the theorem due to de Branges (see  \cite[VI.H.1]{Koosis}) and the second theorem in the paragraph  VI.H.2 of the same book,
 we deduce that there exists a sequence of polynomials $p_m $ satisfying the relation 
 $$
 \|p_m-\omega\|_W
 =\sup\limits_{x\in\mathbb R}\frac{|p_m(x)-\omega (x)|}{W(x)}\to 0\quad \text{as}\quad  m\to+\infty .
 $$ 
Taking into account the inequality (\ref {fi-ab-2}), we obtain that
$$
|p_m(x)\varphi (x)|\le C (1+|x|)^2,\quad x\in \mathbb R, \quad m=1,2, \dots
$$ 
These estimates, together with the Phragmen-Lindel\"of principle and the topological properties  of the space
$\mathcal P(a;b)$, lead to the boundedness of the sequence $p_m\varphi$ in $\mathcal P(a;b)$.
By  \cite[Corollary 3]{Seb}, 
we conclude that there is a subsequence $p_{m_j}\varphi ,$ which converges to the function $s$ in the space $\mathcal P(a;b).$

\bigskip

In the following theorem we obtain sufficient conditions of  weak localizability for a submodule  $\mathcal J\subset \mathcal P(a;b)$ without a priori requirement  of its stability.

\begin{theorem}
Let $\mathcal J\subset \mathcal P(a;b)$ 
be a closed subspace.
Assume that the intervals $(a;b)$,  $(c_{\mathcal J};d_{\mathcal J})$ are unbounded,
and for any $\varphi\in\mathcal J$ the implication
\begin{equation}
e^{\text{i}\, hz}\varphi\in\mathcal P(a;b),
\ \ h\in\mathbb R \Longrightarrow
e^{\text{i}\, hz}\varphi\in\mathcal J 
\label{transs}
\end{equation}
is true.
Then, $\mathcal J$ is a weakly localizable submodule.
\label{teorema3}
\end{theorem}

\begin{proof}
Without loss of  generality, we suppose that $$a=0,\ \  b=d_{\mathcal J}=+\infty .$$

Let us notice that
$$
 \frac{e^{\text{i}\, \tau z}\varphi-\varphi }{\tau} \to z\varphi\quad \text{as}\quad\tau\to 0
$$
in the topology of $\mathcal P(a;b).$ 
Together with the implication (\ref{transs}), it leads  to the conclusion  that $\mathcal J$ is a submodule.
By Corollary \ref{sled1},  we need only to prove the stability property for this submodule .

It is known that the stability of a (closed) submodule at any point $\lambda\in\mathbb C$ follows from its stability at one point $\lambda_0$ (see \cite[Proposition 4.2, remark 1, $\S 4$]{IF-loc-2}).
Let us fix a point $\lambda_0\in\mathbb C\setminus\Lambda_{\mathcal J}  $ 
and show that $\frac{\psi}{z-\lambda_0}\in\mathcal J$ for any function $\psi\in\mathcal J$ vanishing at the point $\lambda_0$. 
According to the definition of the zero set $\Lambda_{\mathcal J}$, we can find a function $\varphi_0\in\mathcal J $
with the property $\varphi_0(\lambda_0)=1.$

Let $S=\mathcal F^{-1} \left(\psi\right),$ $S_{\lambda_0}
=\mathcal F^{-1} \left(\frac{\psi}{z-\lambda_0}\right),$ $S_0=\mathcal F^{-1} \left(\varphi_0\right),$
 $\tilde{S}_{0}=\mathcal F^{-1} \left(\frac{\varphi_0-1}{z-\lambda_0}\right)$. 

Then, for the distributions 
$$
S_1=\mathcal F^{-1} \left(\frac{\varphi_0-1}{z-\lambda_0}\psi\right),
\quad
S_2 =\mathcal F^{-1} \left(\frac{\psi}{z-\lambda_0}\varphi_0\right)
$$
 we have 
   \begin{equation}
   (S_1,f)=(\tilde{S}_0,S*f),\quad
   (S_2,g)=(S_{\lambda_0},S_0*f),
   \label{conv-form}
   \end{equation}
   where $(S*f)(\tau)=(S, f(t+\tau ))$, $f\in \mathcal E (0;+\infty ).$

Let $W\subset \mathcal E (0;+\infty )$ be  the subspace which  annihilator submodule is $\mathcal J .$
   Then,  (\ref{transs}) and (\ref{conv-form})   imply that
the relations
$$
(S_1,f)=0,\ \ (S_2, f)=0,\quad \text{for any } \quad f\in W.
$$
In what follows, $S_1,$ $S_2 \in W^0$ or, equivalently,
$$
\frac{\varphi_0-1}{z-\lambda_0}\psi\in\mathcal J,\quad 
\frac{\psi}{z-\lambda_0}\varphi_0\in\mathcal J.
$$
From the last relations we conclude that
$$
\frac{\psi}{z-\lambda_0}
=\frac{\psi}{z-\lambda_0}\varphi_0
-\frac{\varphi_0-1}{z-\lambda_0}\varphi_0\in\mathcal J.
$$

\end{proof}

\medskip

\section{Dual assertions on $D$-invariant subspaces}

In this section we  obtain the assertions on closed subspaces
$W\subset \mathcal E (a;b),$ which are equivalent  to the theorems proved in the  previous one.

 From here and thereon, unless otherwise specified,
we denote by $W$ a closed $D$-invariant subspace of $\mathcal E(a;b)$ with discrete spectrum $\sigma_W=-\text{i}\, \Lambda,$ where
$\Lambda=\{ (\lambda_j;m_j)\}.$
 
\begin{theorem}  For the subspace  $W$  the representation
\begin{equation}
W=\overline{W_{I_W}+\mathcal L (\text{Exp}\, W)} 
\label{s-W}
\end{equation}
holds (in the other words, $W$ admits weak spectral synthesis) if and only if
its annihilator submodule $\mathcal J$ contains a function $\varphi$ with the property
$$
\mathcal J(\varphi )\subset\mathcal J.
$$
\label{teorema7}
\end{theorem}

\begin{proof}
By \cite[Proposition 3.1]{Al-Kor}, the annihilator submodule $\mathcal J=\mathcal F(W^0)$ is stable.
Proposition \ref{prop1} means that the representation (\ref{s-W}) holds if and only if the submodule  $\mathcal J$
is weakly localizable. In what follows, Theorem \ref{teorema1} leads to the required assertion.
\end{proof}

The following theorem is dual  to Theorem \ref{teorema2}.

\begin{theorem}
Let the subspace $W$  satisfy one of the  assumptions listed below. 

1) The annihilator submodule $\mathcal J=\mathcal F(W^0)$ contains a function $\varphi$ generating weakly localizable principal submodule $\mathcal J_{\varphi}.$

Or,

2) the completeness radius  $\rho_{\Lambda}$ is less than a half of the length of the interval $I_W. $

Then    $W$ admits weak spectral synthesis.   
\label{teorema8}
\end{theorem}

\begin{corollary}
The subspace $W$  admits classical spectral synthesis (\ref{f2}) if and only if
 $I_W=( a;b).$
\label{cor-d-2}
\end{corollary}

For 
any $A,\ B \subset\mathbb R$ we denote by $A\div B$ their {\it geometric difference},
which is equal to the set of all $x\in\mathbb R $ such that  $x+B\subset A.$
Let $S\in \mathcal E'(a;b)$ and $h\in (a;b)\div\text{ch}\, \text{supp}\, S$
(by $\text{ch}\, \text{supp}\, S$ we denote the convex hall of  $\text{supp}\, S$).
We define {\it $h$-translation } $S_h$ 
 setting $(S_h,f)=(S, f(t+h)) $ for any $f\in \mathcal E (a;b).$ 

\begin{theorem}
Let $W\subset \mathcal E(a;b)$ be  an  {\sl arbitrary} closed subspace, and let both  intervals $(a;b)$, $I_W$ 
be unbounded.
Assume that the annihilator subspace $W^0$ is invariant with respect to the translation operator, that is,
the inclusion $S\in W^0$
implies that $ S_h\in W^0$  for any $h\in  (a;b)\div\mathrm{ch}\, \mathrm{supp}\, S.$

Then, the subspace $W$ is  $D$-invariant and admits weak spectral synthesis.
\label{teorema9}
\end{theorem}

\begin{proof}
The assumption on the annihilator subspace $W^0$ implies that the subspace $W$
is $D$-invariant, and the annihilator submodule $\mathcal J=\mathcal F(W^0)$ satisfies the implication (\ref{transs}). By Theorem \ref{teorema3}, this submodule is weakly localizable.  
Therefore, the subspace $W$
admits weak spectral synthesis (Proposition \ref{prop1}).   
 \end{proof}

\bigskip

Natalia Fairbakhovna Abuzyarova,
 
Baskir State University,

 Zaki Validi str., 32

 450074, Ufa, Russia

abnatf@gmail.com


\bigskip

\begin{thebibliography}{99}

\bibitem{Euler} L. Euler  {\it De integratione aequationum differentialum altiorum gradum.}// Miscellanea
Berol. 1743. No. 7. Pp. 193–242.


\bibitem{Schw-1} L. Schwartz {\it “Theorie g\'en\'erale des fonctions moyenne-p\'eriodique} Ann. of Math. 1947. V. 48. No. 4. Pp. 857–929.

\bibitem{Al-Kor}  A. Aleman, B. Korenblum. {\it Derivation-Invariant Subspaces of $C^{\infty}$.}// Computation Methods and Function Theory.  2008. V. 8. No. 2. Pp. 493-512. 

\bibitem{ABB} A. Aleman, A. Baranov, Yu. Belov. {\it Subspaces of $C^{\infty}$ invariant under the differentiation.}// {\tt arXiv:1309.6968v2 [math.CV]}

\bibitem{BM}  A. Beurling, P. Malliavin. {\it On the closure of characters and the zeros of entire functions.}// Acta Math. 1967. V. 118. No. 1-4. Pp. 79-93. 

\bibitem{NF-UMJ-2} N.F. Abuzyarova. {\it Some properties of principal submodules in the  module of entire functions of  exponential type and polynomial growth
 on the real axis.}// Ufa Math. J. 2016. V. 8. No. 1 (to appear).

\bibitem{NF-DAN} N.F. Abuzyarova. {\it Spectral synthesis in the Schwartz space of infinitely differentiable functions.}//
 Doklady Mathematics. 2014. V. 90. No. 1. Pp. 479–482.


\bibitem{IF-synth} I.F. Krasichkov-Ternovskii. {\it Invariant subspaces of analytic functions. I. Spectral analysis on convex
regions.}// Matem. Sbornik. 1972. V.87(129). No. 4. Pp.  459–489  [Math. USSR-Sbornik. 16:4, 471–500
(1972).]


\bibitem{Ehren-synth} L. Ehrenpreis. {\it Mean periodic functions.} Amer. J. Math. 1955. V. 77. No. 2. Pp. 293 —326.


\bibitem{Seb} 
J. Sebastian-e-Silva. {\it On some classes of locally convex spaces important in applications.} //
Matematika. Sbornik Perevodov. 1957. 1:1. Pp. 60-77 (in Russian).

\bibitem{Horm} L. H\"ormander. {\it The analysis of linear partial differential operators I: distribution theory and
Fourier analysis.} Springer. Berlin. 1990.

\bibitem{Levin}  B. Y. Levin (in collaboration with Yu. Lyubarskii, M. Sodin, V. Tkachenko). {\it Lectures on entire functions} (Rev. Edition). AMS. Providence. Rhode Island, 1996. 254 p.

\bibitem{NF-UMJ-1} N.F. Abuzyarova. {\it Closed submodules in the module of entire functions of exponential type and polynomial
growth on the real axis.} Ufa Math. J. 2014. V. 6. No. 4. Pp. 3-17.




\bibitem{IF-loc-1} I.F. Krasichkov-Ternovskii.  {\it Local description of closed ideals and submodules of analytic functions
of one variable. I} // Izvestia AN SSSR. Ser. Matem. 1979. V. 43. No.1. Pp. 44–66. [Math. USSR-Izvestiya.
14:1, 41-60 (1980).]

\bibitem{IF-loc-2} I.F. Krasichkov-Ternovskii.  {\it Local description of closed ideals and submodules of analytic functions
of one variable. II} // Izvestia AN SSSR. Ser. Matem. 1979. V. 43. No. 2. Pp. 309–341. [Math. USSRIzvestiya.
14:2, 289–316 (1980).]


\bibitem{IF-synth-3}  I.F. Krasichkov-Ternovskii. {\it Invariant subspaces of analytic functions. III. On the extension of spectral synthesis.}// Matem. Sbornik. 1972. V.88(130). No. 3. Pp.  331–352.  [Math. USSR-Sbornik. 17:3, 327-348
(1972).] 


\bibitem{Ber-Tayl}  C.A. Berenstein, B.A. Taylor. {\it A new look at interpolation theory 
for  entire functions of  one variable.}// Advances in Mathematics. 1980. V. 33. P. 109-143.



\bibitem{Koosis} P. Koosis. {\it The logarithmic integral I.} Cambridge Univ. Press. 1998. 606 pp.















\end{thebibliography}
\end{document}